\documentclass[12pt]{amsart} 
 \usepackage{amssymb, latexsym, amsmath}

 \usepackage{amssymb, latexsym, amsmath} 
 \begin{document} 
 \theoremstyle{plain} 
 \newtheorem{theorem}{Theorem}[section] 
 \newtheorem{lemma}[theorem]{Lemma} 
 \newtheorem{corollary}[theorem]{Corollary} 
 \newtheorem{proposition}[theorem]{Proposition}

\theoremstyle{definition} 
\newtheorem*{definition}{Definition}
\newtheorem{example}[theorem]{Example}
\newtheorem{remark}[theorem]{Remark}
\title[On the Galois correspondence]{On the 
Galois correspondence for Hopf Galois  structures}

\author{Lindsay N. Childs}
\address{Department of Mathematics and Statistics\\
University at Albany\\
Albany, NY 12222}
\email{lchilds@albany.edu}
\date{\today}
\newcommand{\Zpn}{\mathbb{Z}/p^n\mathbb{Z}}
\newcommand{\NN}{\mathbb{N}} 
\newcommand{\QQ}{\mathbb{Q}} 
\newcommand{\RR}{\mathbb{R}} 
\newcommand{\CC}{\mathbb{C}} 
\newcommand{\FF}{\mathbb{F}} 
\newcommand{\ZZ}{\mathbb{Z}}
\newcommand{\ZZm}{\mathbb{Z}/m\mathbb{Z}}
\newcommand{\ZZp}{\mathbb{Z}/p\mathbb{Z}}
\newcommand{\T}{\Theta} 
\newcommand{\af}{\alpha} 
\newcommand{\bt}{\beta} 
\newcommand{\ep}{\epsilon} 
\newcommand{\ee}{\end{eqnarray}} 
\newcommand{\no}{\noindent} 

\newcommand{\ben}{\begin{eqnarray*}} 
\newcommand{\een}{\end{eqnarray*}} 
\newcommand{\dis}{\displaystyle} 
\newcommand{\beal}{\[ \begin{aligned}} 
\newcommand{\eeal}{ \end{aligned} \]} 
\newcommand{\sg}{\sigma} 
\newcommand{\eps}{\varepsilon} 
\newcommand{\lb}{\lambda} 
\newcommand{\gm}{\gamma} 
\newcommand{\dt}{\delta} 
\newcommand{\zt}{\zeta} 
\newcommand{\Gm}{\Gamma} 
\newcommand{\bpm}{\begin{pmatrix}} 
\newcommand{\epm}{\end{pmatrix}} 
\newcommand{\Fp}{\mathbb{F}_p} 
\newcommand{\Fpx}{\mathbb{F}_p^{\times}} 
\newcommand{\olx}{\overline{x}} 
\newcommand{\oly}{\overline{y}} 
\newcommand{\ols}{\overline{s}} 
\newcommand{\olu}{\overline{u}} 
\newcommand{\olm}{\overline{\mu}} 
\newcommand{\olw}{\overline{w}}
\newcommand{\rbar}{\overline{r}} 
 \newcommand{\ole}{\overline{e}} 
 \newcommand{\olf}{\overline{f}} 
 \newcommand{\olg}{\overline{g}} 
 \newcommand {\olaf}{\overline{\alpha}}
\newcommand{\mcF}{\mathcal{F}} 
\newcommand{\mcp}{\mathcal{p}}
\newcommand{\mbF} {mathbf{F}}
\newcommand{\mcG}{\mathcal{G}}
\newcommand{\mcD}{\mathcal{D}}
\newcommand{\mcP}{\mathcal{P}}
\newcommand{\mcQ}{\mathcal{Q}}
\newcommand{\mcU}{\mathcal{U}}
\newcommand{\mcS}{\mathcal{S}}
\newcommand{\olz}{\overline{z}}
\newcommand{\olr}{\overline{r}}
\newcommand{\olt}{\overline{t}}
\newcommand{\olv}{\overline{v}}
\newcommand{\mcH}{\mathcal{H}} 
\newcommand{\hb}{\hat{b}} 
\newcommand{\ha}{\hat{a}}
\newcommand{\hc}{\hat{c}}
\newcommand{\z}{\hat{\zeta}} 
\newcommand{\noi}{\noindent} 
\newcommand{\hv}{\hat{v}} 
\newcommand{\hp}{\hat{p}} 
\newcommand{\hpi}{\hat{p}_{\psi_i}}
\newcommand{\hpw}{\hat{p}_{\psi_1}}
\newcommand{\hpe}{\hat{p}_{\psi_e}}
\newcommand{\hpd}{\hat{p}_{\psi_d}}
\newcommand{\NBP}{Norm_B(\mcP)}
\newcommand{\mcR}{\mathcal{R}}
\newcommand{\hpx}{\hat{p}_{\chi}}
\newcommand{\Af}{\text{Aff} }
\newcommand{\tb}{\textbullet \ }
\newcommand{\Aff}{\Af_n(\Fp) }
\newcommand{\mcO}{\mathcal{O}}
\newcommand{\GL}{\mathrm{GL}}
\newcommand{\mbx}{\mathbf{x}}
\newcommand{\mby}{\mathbf{y}}
\newcommand{\mfOK}{\mathfrak{O}_K}
\newcommand{\mfOL}{\mathfrak{O}_L}

\newcommand{\mfA}{\mathfrak{A}}

\newcommand{\M}{\mathrm{M}}
\newcommand{\Aut}{\mathrm{Aut}}
\newcommand{\End}{\mathrm{End}}
\newcommand{\Perm}{\mathrm{Perm}}
\newcommand{\Hol}{\mathrm{Hol}}
\newcommand{\Sta}{\mathrm{Sta}}
\newcommand{\GO}{\mathrm{GO}}
\newcommand{\PGL}{\mathrm{PGL}}
\newcommand{\diag}{\mathrm{diag}}
\newcommand{\Reg}{\mathrm{Reg}}

 \begin{abstract} We study the question of  the surjectivity  of the Galois correspondence from subHopf algebras to subfields given by the Fundamental Theorem of Galois Theory for abelian Hopf Galois structures on a Galois extension of fields with Galois group $\Gm$, a finite abelian $p$-group.  Applying the connection between regular subgroups of the holomorph of a finite abelian $p$-group $(G, +)$ and associative, commutative nilpotent algebra structures $A$ on $(G, +)$, we show that if $A$ gives rise to a $H$-Hopf Galois structure on $L/K$, then the $K$-subHopf algebras of $H$ correspond to the ideals of $A$.  As applications, we show that if $G$ and $\Gm$ are both elementary abelian $p$-groups, then the only Hopf Galois structure on $L/K$ of type $G$ for which the Galois correspondence is surjective is the classical Galois structure.  Also, if $\Gm$ is elementary abelian of order $p^n$ and $p > n$ there exists an $H$-Hopf Galois structure on $L/K$ for which there are exactly $n+1$ $K$-sub-Hopf algebras (but approximately $np^n$ intermediate subfields for $n$ large).  By contrast, if $\Gm$ is cyclic of order $p^n$, $p$ odd, then for every Hopf Galois structure on $L/K$, the Galois correspondence is surjective.    
\end{abstract}
\maketitle

\section{Introduction}
The Fundamental Theorem of Galois Theory (FTGT) of Chase and Sweedler [CS69] states that if $L/K$ is a $H$-Hopf Galois extension of fields for $H$ a  $K$-Hopf algebra $H$, then there is an injection $\mathcal{F}$ from the set of $K$-sub-Hopf algebras of $H$ to the  set of intermediate fields $K \subseteq E \subseteq L$ given by sending a $K$-subHopf algebra $J$ to $\mathcal{F}(J) = L^J$.  The \emph{strong form} of the FTGT holds if the injection is also a surjection.  For a classical Galois extension of fields with Galois group $\Gm$, the FTGT holds in its strong form.  It is known from [GP87] that if $L/K$ is a (classical) Galois extension with non-abelian Galois group $\Gm$, then there is a Hopf Galois structure on $L/K$ so that  $\mathcal{F}$ maps onto the subfields $E$ of $L$ that are normal over $K$.  So if $\Gm$ is not a Hamiltonian group [Ha59, 12.5], then $L/K$ has a Hopf Galois structure  for which the strong form of the FTGT does not hold.  In particular, the strong form fails extremely for the unique [By04] non-classical Hopf Galois structure on $L/K$ when $\Gm$ is a non-abelian simple group.  

Suppose $L/K$ is a classical Galois extension with Galois group $\Gm$.  If $L/K$ is also $H$-Hopf Galois for some $K$-Hopf algebra $H$, then $L\otimes_K K$ is $L\otimes_K H$-Hopf Galois over $L$, and by [GP87], $L\otimes_K H \cong LN$ for some regular subgroup $N$ of $\Perm(G)$ normalized by $\lb(\Gm)$, the image in $\Perm(\Gm)$ of the left regular representation of $\Gm$ on $\Gm$.  If $G$ is an abstract group and $\af: G \to \Perm(\Gm)$ is an injective homomorphism from $G$ to $\Perm(\Gm)$ whose image $\af(G) = N$, then we say that $H$ has \emph{type} $G$.   

 Implicit in [GP87], and made explicit by Crespo, Rio and Vela ([CRV16], Proposition 2.2), is that the $K$-subHopf algebras of $H$ correspond to the subgroups of $\af(G)$ that are normalized by $\lb(\Gm)$ in $\Perm(\Gm)$.  
 
Nearly all of the examples examining the success or failure of the strong form of the FTGT for a non-classical Hopf Galois structure on a classical Galois extension $L/K$ with Galois group $\Gm$  involve non-abelian groups.  Perhaps the only wholly abelian example of failure in the literature is in [CRV15], 3.1,  where $\Gm \cong C_2 \times C_2$ and $L/K$ has a Hopf Galois structure by $H$, a $K$-Hopf algebra which is a $K$-form of $LC_4$.  Then by classical Galois theory, there are three intermediate subfields between $K$ and $L$, but $LC_4$ has only one intermediate $L$-Hopf algebra, so $H$ can have at most one intermediate $K$-subHopf algebra.  Hence the strong form of the FTGT cannot hold for that Hopf Galois structure.  

Here we assume that $L/K$ is a Galois extension with Galois group an abelian $p$-group $\Gm$ of order $p^n$.  We look at Hopf Galois structures on $L/K$ by $K$-Hopf algebras $H$ of type $G$ where $G$ is abelian, also of order $p^n$.   There is a sequence of correspondences associated to $H$:  

\tb  the abelian $K$-Hopf algebra $H$ corresponds by base change and Galois descent (as in [GP87]) to  the $L$-Hopf algebra $L \otimes_KH$, which equals $LN$ for some regular abelian subgroup $N$ of $\Perm(\Gm)$ normalized by $\lb(\Gm)$; 

\tb  if $\af:  G \to \Perm(\Gm)$ is a regular embedding with image $N$, then $\af$ corresponds to a regular embedding $\bt:  \Gm \to \Hol(G)$, as shown in [By96];

\tb writing $G = (G, +)$,  the regular subgroup $\bt(\Gm)$ of $\Hol(G)$ corresponds to an associative, commutative nilpotent ring structure $A$ on $(G, +)$, as shown in [CDVS06] and [FCC12]. 

Our main result implies that under these correspondences, the $K$-sub-Hopf algebras of $H$  correspond to ideals of $A$.    As a consequence, we show that  for $\Gm \cong G$ an elementary  abelian $p$-group,  the only case where the FTGT holds in its strong form is if $H$ is the classical Galois structure on $L/K$.   

This paper and  [FCC12], [Ch15] and [Ch16]  demonstrate in different ways the usefulness of the correspondence of [CDVS06] in the Hopf Galois theory of Galois extensions of fields whose Galois group is a finite abelian $p$-group.  

\section{Translating to and from the holomorph}

Let $L/K$ be a Galois extension with Galois group $\Gm$ and let $G$ be a group of the same cardinality as $\Gm$.  The $H$-Hopf Galois structures on $L/K$ of type $G$ correspond to images of regular embeddings $\af:  G \to \Perm(\Gm)$ where $\af(G)$ is normalized by $\lb(G)$.   

As shown in [By96], a regular embedding $\af: G \to \Gm$ whose image $\af(G)$ is normalized by $\lb(\Gm)$ corresponds to a regular embedding $\bt:  \Gm \to \Hol(G)$, where 
\[\Hol(G) = \rho(G)\cdot \Aut(G) \subset \Perm(G)\]
 is the normalizer of $\lb(G)$ in $\Perm(G)$.   Here $\rho: G \to \Perm(G)$is the right regular representations of $G$ in $\Perm(G)$.   The relationship between $\af$ and $\bt$  is as follows:
 
 Let $\bt:  \Gm \to \Hol(G)$ be a regular embedding.  Define  $ b:  \Gm \to G $ by
\[ b(\gm) = \bt(\gm)(e_G) \]
for $\gm$ in $\Gm$, $g$ in $G$, where $e_G$ is the identity element of $G$.  Then
\[ \bt(\gm)(g) = (b(\lb(\gm))b^{-1})(g) = (C(b)\lb(\gm))(g) \]
Define $\af:  G \to \Perm(\Gm)$ by 
\[ \af(g)(\gm) = (b^{-1}(\lb(g))b)(\gm) = (C(b^{-1})\lb(g))(\gm).\]

The following formulas will be useful below.

\begin{proposition}\label{ba}   Suppose $\bt: \Gm \to \Hol(\lb(G))$ is an regular embedding,  and let  $\af = C(b^{-1})\lb_G:  G \to \Perm(\Gm)$  be the  regular embedding corresponding to $\bt$.    Then for all $\gm$ in $\Gm$ and $g$ in $G$, there is some $h$ in $G$ so that
\[ \bt(\gm)\lb(g) \bt(\gm)^{-1} = \lb(h) \]
and
\[ \lb(\gm) \af(g) \lb(\gm)^{-1} = \af(h) .\]
\end{proposition}

\begin{proof}  The first formula follows because $\bt$ maps $\Gm$ into $\Hol(G)$, the normalizer of $\lb(G)$ in $\Perm(G)$.
Since $C(b^{-1})(\bt)(\gm) = \lb(\gm)$ and $C(b^{-1})\lb(g) = \af(g)$, the second formula follows from 
 the first by applying $C(b^{-1}) $ to the first formula.  \end{proof}

\section{On the Galois correspondence for Hopf Galois structures}

Let $L/K$ be a Galois extension of fields with Galois group $\Gm$, and suppose $\af:G \to Perm(\Gm)$ is a regular embedding such that $\lb(\Gm)$ normalizes $\af(G)$.  Then by descent, $H = L[\af(G)]^G$ is a $K$-Hopf algebra and there is an action of $H$ on $L$ making $L/K$ into an $H$-Hopf Galois extension.  (See [GP87].)  The Fundamental Theorem of Galois Theory of Chase and Sweedler [CS69] gives an injection $\mathcal{F}$ from the set of $K$-subHopf algebras of $H$ to the set of intermediate fields $F$ with $K \subseteq F \subseteq L$ by 
\[ H' \mapsto L^{H'} = \{ x \in L | hx = \epsilon(h)x  \text{  for all }h \text{ in } H'\} .\]
From Proposition 2.2 of [CRV16], the $K$-subHopf algebras of $H$ correspond to $\lb(\Gm)$-invariant subgroups of $\af(G)$.  Thus to study the image of $\mathcal{F}$, we look at the $\lb(\Gm)$ invariant subgroups of $\af(G)$.

We do this for $G$ a finite abelian $p$-group by utilizing a characterization of regular subgroups of $\Hol(G)$ due to 
 Caranti, Della Volta and Sala [CDVS06] as extended in Proposition 2 of [FCC12]: 

\begin{proposition}  Let $(G, +)$ be a finite abelian $p$-group.  Then each regular subgroup of $\Hol(G)$ is isomorphic to the group $(G, \circ)$ induced from a structure $(G, + , \cdot)$ of a commutative, associative nilpotent ring (hereafter, ``nilpotent'') on $(G, +)$,  where the operation $\circ$ is defined by $g \circ h = g + h + g\cdot h$. \end{proposition}

 The idea is the following: Let $(G, +)$ be an abelian  group of  order $p^n$ , and suppose that $A = (G, + , \cdot)$ is a nilpotent ring structure on $(G, +)$  yielding the operation $\circ$. Define $\tau: (G, \circ) \to \Hol(G, +)$ by $\tau(g)(x) = g \circ x$.  Then $\tau(g)(0) = g$, and 
\[ \tau(g)\tau(g')(x) = \tau(g)(g' \circ x) =g \circ (g '\circ x) = (g \circ g') \circ x = \tau(g \circ g')(x),\]
so $\tau$ is an isomorphism from $(G,\circ)$ into $\Hol(G, +)$ whose image $\tau(G, \circ) = T$ is a regular subgroup of $\Hol(G)$.  This process is reversible:  given a regular subgroup $T$ of $\Hol(G, +)$, there is a nilpotent ring structure $(G, +, \cdot)$ on $G$, which defines the $\circ$ operation as above and  yields a unique isomomorphism $\tau:  (G, \circ) \to T$ so that $\tau(g)(x) = g \circ x$. 

Now suppose $L/K$ be a Galois extension with Galois group $\Gm$, a finite abelian $p$-group of order $p^n$,  and  $\xi:\Gm \to (G, \circ)$ is some isomorphism.  

Given $A$ and $\tau$, let $\bt = \tau \xi:  \Gm \to T$, an embedding of $\Gm$ onto $T$ in $\Hol(G)$.  Then $\bt$ yields the map 
\[ b : \Gm \to G \]
by 
\[ b(\gm) = \bt(\gm)(0) = \tau(\xi(\gm))(0) = \xi(\gm) \circ 0 = \xi(\gm).\]
Thus $b$ is an isomorphism from $\Gm$ to $(G, \circ)$.  
Then $b$ defines an embedding $\af: (G, +) \to\Perm(\Gm)$, and we have the relations of Proposition \ref{ba}:  for all $\gm$ in $\Gm$, $g$ in $G$, there is some $h$ in $G$ so that
\[ \lb(\gm) \af(g) = \af(h) \lb(\gm)  \]
and
\[ \bt(\gm)\lb(g)= \lb(h) \bt(\gm).\]

\begin{theorem}\label{main} Suppose the nilpotent algebra $A = (G, +, \cdot)$ yields the regular embedding $\af:  (G, +) \to \Perm(\Gm)$ whose image is normalized by $\lb(\Gm)$.  Then the lattice (under inclusion) of $\lambda(\Gm)$-invariant  subgroups of $\af(G) $ is isomorphic to the lattice of ideals of $A$.   \end{theorem}

\begin{proof}  First, $\af: G \to Perm(\Gm)$ is an injective homomorphism from $(G, +)$ to $\Perm(\Gm)$.  So if $J$ is an additive subgroup of $G$, then $\af(J)$ is a subgroup of $\af(G) \subset \Perm(\Gm)$.  Conversely, if $J$ is a subset of $G$ and $\af(J)$ is a subgroup of $\af(G)$, then for $s, t$ in $J$, $\af(s + t) = \af(s)\af(t)$ is in $\af(J)$, so $J$ is an additive subgroup of $G$.   So there is a bijection between subgroups of $(G, +)$ and subgroups of $\af(G)$.  Clearly $J_1 \subseteq J_2$ iff $\af(J_1) \subseteq \af(J_2)$, so the bijection is lattice-preserving.

Suppose the image $\af(G)$ of $\af$ is normalized by $\lb(\Gm)$, so for all $\gm$ in $\Gm$, $g$ in $G$, there is some $h$ in $G$ so that
\[ \lb(\gm)\af(g)\lb(\gm)^{-1} = \af(h).\]
This equation holds iff
\[ \bt(\gm)\lb_G (g) = \lb_G(h) \bt(\gm).\]
Recalling that $A = (G, +, \cdot) = (G, \circ)$, factor $\bt = \tau\xi$ where $\xi:  \Gm \to A = (G, \circ)$ is an isomorphism and $\tau:  A = (G, \circ) \to Hol(G)$ an embedding. 
Let $\xi(\gm) = k$ in $A$.  Then the last equation is
\[\tau(k)\lb_G(g) = \lb_G(h)\tau(k),\]
and applying this to $x$ in $G$ gives
\[ \tau(k)(g + x) = h + \tau(k)(x) ,\]
\[ k \circ (g + x) = h + k \circ x .\]
Viewing this equation in $A$ where $a \circ b = a + b + a\cdot b$, we have
\[k + (g + x) + k\cdot g + k \cdot x = h + k + x + k\cdot x.\]
 This last equation reduces to 
\[ h = g + k\cdot g.\]

Now suppose $J$ is an ideal of $A$ and $g$ is in $J$.  Then $k\cdot g$ is in $J$, so $h$ is in $J$, and so $\lb(\gm)$  conjugates $\af(g)$ in $\af(J)$ to an element of $\af(J)$.  So $\af(J)$ is normalized by $\lb(\Gm)$ in $\Perm(\Gm)$.  

Conversely, suppose $J$ is an additive subgroup of $(G, +, \cdot) = A$ and $\af(J)$ is normalized by $\lb(\Gm)$.  Then for all $\gm$ in $G$, $g$ in $J$, 
\[ \lb(\gm)\af(g)\lb(\gm)^{-1} = \af(h)\]
and $\af(h)$ is in $\af(J)$.  So $h$ is in $J$.  Hence for all $k = \xi(\gm)$ in $G$, and $g$ in $J$, $h = g + k\cdot g$ is in $J$.  Now $J$ is an additive subgroup of $A$, so $k\cdot g$ is in $J$ for all $k$ in $G$, $g$ in $J$.  Thus $J$ is an ideal of $A$.
\end{proof} 

\section{Examples}

Theorem \ref{main} transforms the problem of describing the image of the Galois correspondence map $\mathcal{F}$ on a $H$-Hopf Galois structure on $L/K$ to the study of the ideals of the nilpotent algebra associated to $H$.   In this section we look at some examples.  

\begin{theorem}  Let $L/K$ be a Galois extension of fields with  Galois group $\Gm$ an elementary  abelian $p$-group of order $p^n$.   Suppose $L/K$ has a Hopf Galois structure by an abelian Hopf algebra $H$ of type $G$ where $G$ is an elementary abelian $p$-group. Let $A$ be a nilpotent ring structure yielding the regular subgroup $T \cong (G, \circ) \subset \Hol(G)$  corresponding to $H$, where $(G, \circ) \cong \Gm$.  Then the $H$-Hopf Galois structure on $L/K$ satisfies the strong form of the FTGT if and only if $A$ is the trivial nilpotent algebra satisfying $A^2 = 0$, if  and only if $T = \rho(G)$, iff  $H$ is the classical Galois structure by $K\Gm)$ on $L/K$.  \end{theorem}

\begin{proof}  If $A^2 = 0$, then $(G, \circ) = (G, +)$, so the regular subgroup $T$ acts on $G$ by $\tau(g)(h) = g \circ h = g + h$, hence $T = \lb(G)$. Since $G$ is abelian, the corresponding Hopf Galois structure on $L/K$ is the classical structure by the $K$-Hopf algebra $K[\Gm]$.  So the Galois correspondence holds in its strong form.  

 For the converse, view $(G, +)$ as an $n$-dimensional $\Fp$-vector space.  Suppose $A^2 \ne 0$.    Then for some $a, b$ in $A$,  $ab \ne 0$.  Then the subspace $\Fp a$ does not contain $ab$.  For if $ab = ra$ for $r \ne  0$ in $\Fp$, then $a = sba$ for $s \ne 0$ in $\Fp$.  Then
 \[ a = (sb)a = (sb)^2a = \ldots = (sb)^{n+1}a = 0 \]
 since $A$ is nilpotent of  dimension $n$,  hence $(sb)^{n+1}= 0$.  Thus the subspace $\Fp a$  is not an  ideal of $A$.  

The  subgroup $\af(\Fp a)$ of $\af(G)$ is then not normalized by $\lb(\Gm)$.  But $\Gm \cong G$, so there are bijections between  subgroups of $\af(G)$, subgroups of $G$,  subgroups of $\Gm$ and (by classical Galois theory) subfields of $L$ containing $K$.  If some subgroup of $\af(G)$ is not normalized by $\lb(\Gm)$, then the number of $K$- subHopf algebras of $H =L[\af(G)]^G$ is strictly smaller than the number of subfields between $K$ and $L$.  So the Galois correspondence for the $H$-Hopf Galois structure on $L/K$ does not hold in its strong form.
\end{proof}

By choosing a particular nilpotent algebra structure on $(\Fp^n, +)$ we can see how badly the Galois correspondence can fail to be surjective. 

Let  $A$ be the primitive $n$-dimensional nilpotent $\Fp$-algebra generated by $z$ with $z^{n+1} = 0$.   Then $(A, +) \cong (\Fp^n, +)$ and so the multiplication on $A$ yields a nilpotent $\Fp$-algebra structure on $(G, +) = (\Fp^n, +)$.  Let $G = (\Fp^n, \circ)$ where the operation $\circ$ is defined using the multiplication on $A$ by $a \circ b = a + b + a\cdot b$.  

\begin{theorem} \label{prim} 
Let $G$ be an elementary  abelian $p$-group of order $p^n$.   Let $A$ be a primitive $\Fp$-algebra structure $A$ on $G$, and let $(G, \circ)$ be the corresponding group structure on $\Fp^n$.  Suppose $L/K$ is a Galois extension of fields with  Galois group $\Gm \cong (G, \circ)$.  Then the primitive nilpotent $\Fp$-algebra $A$ corresponds to an $H$-Hopf Galois structure on $L/K$  for some $K$-Hopf algebra $H$, and the  $K$-subHopf algebras of $H$ form a descending chain
 \[ H = H_1 \supset H_2 \supset \ldots \supset H_n \supset K .\]
 Hence the Galois correspondence $\mathcal{F}$ for $H$ maps onto exactly $n+1$ fields $F$ with $K \subseteq F \subseteq L$.  
\end{theorem}

\begin{proof}  
Given Theorem \ref{main}, we just need to show that ideals of $A$ are $J_i = \langle z^i \rangle$ for $i = 1, \ldots, n$.  

Suppose $J$ is an ideal of $A$ and has an element  $r(z^k + z^{k+r}b)$ of minimal degree $k$, where $r \ne 0$ in $\Fp$, $b$ in $A$. Then $J$ also contains
\[ z^k + z^{k+r}b\] 
and
\[ (z^k + z^{k+r}b) (-z^rb)  = -z^{k+r}b -z^{k + 2r}b^2,\]
hence their sum,
\[z^k - z^{k+2r}b^2  = z^k + z^{k + r'}b'\]
where $r' > r$.  Repeating this argument until $r' > n$ shows that $J$ contains $z^k$, hence $J \supseteq J_k =\langle z^k \rangle$.  
Since $J_k = \langle z^k \rangle$ contains every element of degree $\ge k$, $J = J_k$.  
Thus $A$ has exactly $n+1$ ideals.  Since the correspondence between ideals of $A$ and $\lb(\Gm)$ invariant subgroups of $\af(G)$ is lattice-preserving, 
 we have a single filtration
\[ \af(G) = \af(J_1) \supset \af(J_2) \supset  \ldots  \supset   \af(J_n)  \supset  {0} .\]
of $\lb(G)$-invariant subgroups of $\af(G)$.  If $H$ is the corresponding $K$-Hopf algebra making $L/K$ into a Hopf Galois extension, then $H$ has a unique filtration of $K$-sub-Hopf algebras, 
\[ H = H_1 \supset H_2 \supset  \ldots \supset H_n \supset K .\]
\end{proof}
 
For $A$ a primitive nilpotent $\Fp$-algebra with $A^{n+1} = 0$, the corresponding group $(G, \circ)$ is isomorphic (by $a \mapsto 1 + a$) to the group of principal units of the truncated polynomial ring $\Fp[x]/(x^{n+1}\Fp[x]$.  The structure of that group is described in Corollary 3 of [Ch07].  In particular $(G, \circ)$, hence $\Gm$, is an elementary abelian $p$-group if and only if $p > n$.

In  Theorem \ref{prim}, when  $p > n$, then  $L/K$ is classically Galois with Galois group $\Gm \cong  (\Fp^n, +)$.  So the number of subgroups of $\Gm$, and hence the the number of subfields $E$ with $K \subseteq E \subseteq L$, is equal to the number of subspaces of $\Fp^n$, namely
\[ \sum_{r=1}^n \frac{ (p^n-1)(p^n - p) \cdots (p^n-p^{r-1})}{(p^r-1)(p^r-p) \cdots (p^r - p^{r-1})} \sim np^n\]
for $n$ large.   So the Galois correspondence map $\mathcal{F}$  is extremely far from being surjective for a Hopf Galois structure corresponding to a nilpotent algebra structure $A$ with $\dim(A/A^2) = 1$.

By contrast, 

\begin{proposition} Let $L/K$ be a Galois extension of fields with Galois group $\Gm$ cyclic of order $p^n$, $p$ odd.  Let the $K$-Hopf algebra $H$ give a Hopf Galois structure on $L/K$.  Then $H$  has type $G$ where $G \cong \Gm$, and the Galois correspondence for that Hopf Galois structure holds in its strong form.  \end{proposition}

\begin{proof} From [Ko98] it is known that if $\Gm$ is cyclic of order $p^n$ then every Hopf Galois structure must have type $G \cong \Gm$.  So let $G$ be cyclic of order $p^n$, which we identify with $(\ZZ/p^n\ZZ, +)$.  Then we  view $\Hol(G) = G \rtimes \Aut(G)$ as the set of pairs $(a, g)$ where $a$ and $g$ are modulo $p^n$ and $(g, p) = 1$, or equivalently as the set of matrices
\[ \bpm  g & a\\0&1 \epm \]
in $\Af_1(\ZZ/p^n\ZZ)  \subset \GL_2(\ZZ/p^n\ZZ)$, acting on $s$ in $G$ by 
\[ \bpm g&a\\0&1 \epm \bpm s\\1 \epm = \bpm gs + a \\1\epm.\]
View $\Gm$ as the free $\ZZ/p^n\ZZ$-module with basis $y$.   From Proposition 2 of [Ch11],  the $p^{n-1}$ regular embeddings $\bt:  \Gm = (\ZZ/p^n\ZZ) y \to Hol(G)$ are determined by $\bt(y)$ where
 \[ \bt(y) = \bpm 1 + pd & -1\\0&1 \epm \]
for some $d$ modulo $p^{n-1}$.  So in the notation of the last section,  
\[ \tau(-1) = \bpm(1 + pd & -1\\0&1 \epm\]
and acts on $s$ in $G$ as above.  That action defines the operation $\circ$ on $G$ by 
\[ (-1) \circ s = (1 + pd)s - 1 = -1 +s + pds. \]
The multiplication on $(G, + )$ to make $(G, + , \cdot) = A$ a nilpotent algebra is then defined by 
\[(-1) \cdot s = (-1) \circ s -((-1) + s) = (-1 + s +  pds) +1 -s =  pds.\]
By distributivity, for every $r, s$ in $\ZZ/p^n\ZZ$,
\[ -r \cdot s = rspd.\]
Replacing $d$ by $-d$, let $A_d$ be the commutative nilpotent algebra structure on $(\ZZ/p^n\ZZ, +)$ with multiplication 
\[ r \cdot s = rspd \] 
for all $r, s$ in $\ZZ/p^n\ZZ$.   It is then easy to check that the ideals of $A_d$ are the  principal ideals generated by $p^r$, for $r = 0, \ldots, n$.  Since those are also the additive subgroups of $(A_d, +) =(\ZZ/p^n\ZZ, +)$, it follows by Theorem \ref{main} that for every Hopf Galois structure on $L/K$, the Galois correspondence holds in its strong form.
\end{proof}

\end{document}